\documentclass[letterpaper, 10 pt, conference]{ieeeconf}  

\IEEEoverridecommandlockouts 

\usepackage{afterpage}
\usepackage{epsfig} 
\usepackage{times} 
\usepackage{amsmath} 
\usepackage{amssymb}  
\usepackage{amsfonts}
\usepackage{mathptmx} 
\usepackage{tabularx}
\usepackage{epstopdf}
\usepackage{float}
\usepackage{epic,color}
\usepackage{mathrsfs}
\usepackage{dsfont,MnSymbol}
\usepackage{algorithm}
\usepackage{multirow} 
\usepackage[T1]{fontenc}

\newcounter{rmnum}

\newcounter{anum}

\def\IEEEQEDclosed{\mbox{\rule[0pt]{1.3ex}{1.3ex}}}

\def\qed{\ifmmode\IEEEQEDclosed\else{\unskip\nobreak\hfil
		\penalty50\hskip1em\null\nobreak\hfil\IEEEQEDclosed
		\parfillskip=0pt\finalhyphendemerits=0\endgraf}\fi}

\def\qed{\hspace*{\fill}~\IEEEQED\par\endtrivlist\unskip}

\def\Re{\mathbb{R}}

\def\Sec#1{Sec.~\ref{#1}}

\def\notes#1{\marginpar{\tiny #1}\typeout{Notes!
Notes!
Notes!
}}
\renewcommand{\notes}[1]{\typeout{notes!}}
\def\FRAC#1#2#3{\genfrac{}{}{}{#1}{#2}{#3}}
\def\half{{\mathchoice{\FRAC{1}{1}{2}}%
{\FRAC{2}{1}{2}}%
{\FRAC{3}{1}{2}}%
{\FRAC{4}{1}{2}}}}

\def\Re{\field{R}}

\def\Sec#1{Sec.~\ref{#1}}

\def\transpose{{\hbox{\rm\tiny T}}}

\def\clE{{\cal E}}

\def\clP{{\cal P}}
\def\clZ{{\cal Z}}

\def\Sec#1{Sec~\ref{#1}}

\def\E{{\sf E}}

\def\Sec#1{Sec.~\ref{#1}}

\def\IEEEQEDclosed{\mbox{\rule[0pt]{1.3ex}{1.3ex}}}
\def\qed{\nobreak\hfill\IEEEQEDclosed}

\def\clZ{{\cal Z}}

\newtheorem{example}{Example}
\newtheorem{definition}{Definition}
\newtheorem{lemma}{Lemma}
\newtheorem{remark}{Remark}
\newtheorem{proposition}{Proposition}

\def\beq{\begin{eqnarray}} 
\def\bc{\begin{center}} 
\def\be{\begin{enumerate}}
\def\bi{\begin{itemize}} 
\def\bs{\begin{small}}
\def\bS{\begin{slide}}
\def\ec{\end{center}} 
\def\ee{\end{enumerate}}
\def\ei{\end{itemize}}
\def\es{\end{small}}
\def\eS{\end{slide}}
\def\eeq{\end{eqnarray}}

\newcommand{\newP}[1]{\medskip\noindent{\bf #1:}}

\newcommand{\ud}{\,\mathrm{d}}

\def\Re{\mathbb{R}}
\def\E{{\sf E}}

\def\clY{{\cal Y}}

\def\Sec#1{Sec.~\ref{#1}}

\def\Prop#1{Prop.~\ref{#1}}

\def\clD{{\cal D}}
\def\clE{{\cal E}}

\def\clP{{\cal P}}
\def\clZ{{\cal Z}}

\renewcommand{\Re}{\mathbb{R}}

\def\FRAC#1#2#3{\genfrac{}{}{}{#1}{#2}{#3}}

\newcommand{\var}{\text{var}}

\def\clA{{\cal A}}

\def\clD{{\cal D}}
\def\clE{{\cal E}}
\def\clF{{\cal F}}

\def\clN{{\cal N}}
\def\clO{{\cal O}}
\def\clP{{\cal P}}

\def\clS{{\cal S}}

\def\clV{{\cal V}}

\def\clY{{\cal Y}}
\def\clZ{{\cal Z}}
\def\E{{\sf E}}
\def\bS{\mathbb{S}}

\def\ones{{\sf 1}}
\def\sP{{\sf P}}
\def\tsP{{\tilde{\sf P}}}
\def\tE{{\tilde{\sf E}}}

\def\tp{{\hbox{\rm\tiny T}}}

\def\bmu{{\bar\mu}}

\def\chisq{{\chi^2}}

\def\tv{{\text{\tiny TV}}}

\def\bO{\mathbb{O}}

\def\LL{{\sf L}}
\def\II{{\cal E}}
\def\VV{\var^\rho(Y_0(X_0))}

\newlength{\noteWidth}
\long\def\notes#1{\ifinner
	{\tiny #1}
	\else
	\marginpar{\parbox[t]{\noteWidth}{\raggedright\tiny #1}}
	\fi}

\title{\LARGE \bf Backward Map for Filter Stability Analysis}

\author{Jin Won Kim, Anant A. Joshi and Prashant G. Mehta
	\thanks{This work is supported in part by the AFOSR award
		FA9550-23-1-0060, the NSF award 2336137 (Joshi, Mehta) and the DFG grant 318763901/SFB1294 (Kim).
	}
	\thanks{J. W. Kim is with the Department of Mechanical and System Design Engineering at the Hongik University. (e-mail: jin.won.kim@hongik.ac.kr)}
	\thanks{A.~A.~Joshi and P.~G.~Mehta are with the Coordinated
	Science Laboratory and the Department of Mechanical Science and
	Engineering at the University of Illinois at Urbana-Champaign.
	(e-mail: anantaj2; mehtapg@illinois.edu)}
}

\begin{document}

\maketitle
\thispagestyle{empty}
\pagestyle{empty}

\begin{abstract}
	
A backward map is introduced for the purposes of analysis of nonlinear (stochastic) filter stability. The backward map is important because filter-stability, in the sense of $\chisq$-divergence, follows from a certain variance decay property associated with the backward map.  To show this property requires additional assumptions on the hidden Markov model (HMM).  The analysis in this paper is based on introducing a Poincar\'e Inequality (PI) for HMMs with white noise observations.  In finite state-space settings, PI is related to both the ergodicity of the Markov process as well as the observability of the HMM. It is shown that the Poincar\'e constant is positive if and only if the HMM is detectable.  

\end{abstract}

\section{Introduction}
\label{sec:intro}

 
Dissipation is at the heart of any stability theory for dynamical
systems.  For Markov processes, dissipation is referred to as
{\em variance decay}.  To illustrate the key ideas, consider a Feller-Markov
process $X=\{X_t: t\geq 0\}$ taking values in a Polish state-space
$\bS$ and suppose $\bmu$ is a given invariant measure. The
fundamental object of interest is the Markov semigroup defined
by~\cite[Eq.~(1.1.1)]{bakry2013analysis}
\[
(P_tf)(x):=\E^x (f(X_t)),\quad x\in\bS,\quad t\geq 0
\]   
for $f:\bS\to \Re$ in some class of measurable functions.  The problem of
stochastic stability is to show that $P_tf \to \bmu(f)$, in some 
suitable sense, as $t\to\infty$.  As defined, $(P_tf)(x)$ has an
interpretation as the expectation of the random variable $f(X_t)$ starting
from an initial condition $X_0=x$.  Therefore, $P_tf \to \bmu(f)$ means
that this expectation asymptotically converges to its stationary value
$\bmu(f)$ for all choices (in suitable sense) of the initial conditions $x\in\bS$.  


The dissipation equation requires the notation,
\begin{align*}
		\text{(variance)}\qquad\clV^\bmu(f) &:= \bmu (f^2)
                                                       -\bmu(f)^2 \\
		\text{(energy)}\qquad\clE^\bmu(f) &:= \bmu\big(\Gamma f)
\end{align*}
where $\Gamma$ is the so called carr\'e du champ operator (Defn.~\ref{def:cdc}).  The
operator is a positive-definite bilinear form. 
Using these definitions, the dissipation equation arises as
\[
\frac{\ud }{\ud t} \clV^\bmu(P_t f) = - \clE^\bmu(P_t f),\quad t\geq 0
\]
The calculation for the same appears in Appendix~\ref{apdx:diss_MP}, see also~\cite[Thm.~4.2.5.)]{bakry2013analysis}. The equation shows that $\{\clV^\bmu(P_t f) :t\geq 0\}$ is non-increasing. To show that the variance decays to zero requires a suitable relationship between energy and variance.  The simplest such relationship is through the Poincar\'e Inequality (PI):
\[
	\text{(PI)}\qquad  \clE^\bmu(f) \ge c\,
         \clV^\bmu(f),\quad  \forall f  \in \clD
         \]
where the sharpest such constant $c$ is referred to as the Poincar\'e constant; here, $\clD$ is a suitable space of test functions for which the energy $\clE^\bmu(f)$ is well-defined (see Defn.~\ref{def:cdc}).  PI is useful to conclude stochastic stability where $c$ gives the exponential rate of convergence.  For reversible Markov processes, $c$ is referred to as the spectral gap constant.


\subsection{Aims and contributions of this paper}
  
This paper is concerned with extension of variance decay to the study of the nonlinear filter.  Specifically, the following questions are of interest:
\begin{enumerate}
\item[\textbf{Q1.}] What is the appropriate notion of variance decay for a nonlinear
  filter?  And how is it related to filter stability?
\item[\textbf{Q2.}] What is the appropriate generalization of the dissipation
  equation for the nonlinear filter?
\item[\textbf{Q3.}]  What is the appropriate generalization of the Poincar\'e
  inequality for the filter? And how is it 
  related to the hidden Markov model (HMM) properties such as observability and detectability?
\end{enumerate}
In this paper, we provide an answer to each of these questions
(see \Prop{prop:ans_Q1} for Q1, \Prop{prop:backward-map-and-bsde} for Q2,
and \Prop{prop:sufficiency}-\ref{prop:necessity} for Q3).  An
original contribution of this paper is the backward map that is
introduced here for a general class of HMMs. The backward map is important because filter-stability, in the sense of $\chisq$-divergence, follows from a certain variance decay property associated with the backward map.  While the backward map and the variance decay is for a general class of HMMs, the answers to Q2 and Q3 are
given for HMM with white noise observations.
The overall approach may be regarded as an {\em optimal
control approach} to filter stability based on our recent work on
duality~\cite{duality_jrnl_paper_I,duality_jrnl_paper_II,JinPhDthesis}.
Our approach is contrasted with the intrinsic approach to filter
stability based upon specification of a certain forward map~\cite{chigansky2009intrinsic}.

\subsection{Outline of the remainder of this paper}

\Sec{sec:prelim} contains math preliminaries for the HMM and the filter
stability problem.  The backward map is introduced in
\Sec{sec:bmap} and specialized to white noise observations in
\Sec{sec:white-noise-model}. For this HMM, the definition of the
Poincar\'e Inequality (PI) is introduced in \Sec{sec:main-results}.
The PI is related to the HMM model properties in
\Sec{sec:model_props} (for the finite state-space settings) and illustrated using numerics in
Sec.~\ref{sec:numerics}.  
The proofs appear in the Appendix. 

  
\newpage

\section{Mathematical Preliminaries}\label{sec:prelim}

\subsection{Hidden Markov model (HMM)}

For the definition and analysis of the nonlinear filter, a standard
model for HMM is specified (see~\cite[Sec.~2]{van2009observability})
through construction of a
pair of stochastic processes $(X, Z):=\{(X_t,Z_t):0\leq
  t\leq T\}$ on probability space
$(\Omega,\clF_T,\sP)$ as follows: 
\begin{itemize}
\item The {state-space} $\bS$ is a locally compact Polish space.  The
  important examples are (i) $\bS=\{1,2,\hdots,d\}$ of finite or countable
  cardinality, and (ii) $\bS\subseteq \Re^d$. 
\item The {observation-space} $\bO=\Re^m$.
\item The {signal-observation process} $(X,Z)$ is a Feller-Markov process.
\item The {state process} $X$ is a Feller-Markov process with $X_0\sim \mu \in\clP(\bS)$. Here, $\clP(\bS)$ is the space of probability measures defined on the Borel $\sigma$-algebra on $\bS$ and $\mu$ is referred to as the prior.  
\item The observation process $Z$ has $Z_0=0$ and conditionally independent
  increments given the state process $X$.  That is, given $X_t$, an increment
  $Z_s-Z_t$ is independent of $\clZ_t := \sigma\big(\{Z_s:0\le
	s\le t\}\big)$, for all $s>t$.  The filtration generated by the observations is denoted 
	$\clZ :=\{\clZ_t:0\le t\le T\}$. 
\end{itemize}

The objective of nonlinear  filtering is to compute the
conditional expectation
\[
\pi_T(f):= \E\big(f(X_T)\mid \clZ_T\big),\quad f \in C_b(\bS)
\]
where $C_b(\bS)$ is the space of continuous and bounded functions. The conditional measure $\pi_T$ is referred to as the {\em nonlinear filter}.

To stress the dependence on the prior $\mu$, a standard convention is to denote the probability space as $(\Omega,\clF_T,\sP^\mu)$, the expectation operator as $\E^\mu$, and the nonlinear filter as $\pi_T^\mu$.  In practice, the prior may not be known.  With an incorrect choice of prior $\nu\in\clP(\bS)$, the filter is denoted as $\pi_T^\nu$.  The precise meaning for these along with the definition of filter stability appear next.   

\subsection{Filter stability}

Let $\rho \in \clP(\bS)$.  On the common measurable space $(\Omega,
\clF_T)$, $\sP^\rho$ is used to denote another probability measure
such that the transition law of $(X,Z)$ is identical but $X_0\sim
\rho$ (see~\cite[Sec.~2.2]{clark1999relative} for an explicit construction
of $\sP^\rho$ as a probability measure over paths on $\bS\times\bO$.).  The associated expectation operator is denoted by $\E^\rho(\cdot)$ and the nonlinear filter by $\pi_t^\rho(f) =
\E^\rho\big(f(X_t)\mid\clZ_t\big)$.  The
two important choices for $\rho$ are as follows:
\begin{itemize}
	\item $\rho=\mu$. The measure $\mu$ has the meaning of the true prior.
	\item  $\rho=\nu$.  The measure $\nu$ has the meaning of an
          incorrect prior that is used to compute the filter.  
\end{itemize}
The relationship between $\sP^\mu$ and $\sP^\nu$ is as follows ($\sP^\mu|_{\clZ_t}$ denotes
the restriction of $\sP^\mu$ to the $\sigma$-algebra $\clZ_t$):

\medskip

\begin{lemma}[Lemma 2.1 in \cite{clark1999relative}] \label{lm:change-of-Pmu-Pnu}
	Suppose $\mu\ll \nu$. Then 
	\begin{itemize}
		\item $\sP^\mu\ll\sP^\nu$, and the change of measure is given by
		\begin{equation*}\label{eq:P-mu-P-nu}
			\frac{\ud \sP^\mu}{\ud \sP^\nu}(\omega) =
			\frac{\ud \mu}{\ud
				\nu}\big(X_0(\omega)\big)\quad
			\sP^\nu\text{-a.s.} \;\omega
		\end{equation*}
		\item For each $t > 0$, $\pi_t^\mu \ll \pi_t^\nu$, $\sP^\mu|_{\clZ_t}$-a.s..
	\end{itemize} 
	
\end{lemma}

\medskip

Suppose $\mu\ll\nu$.  Then $\pi_T^\mu \ll \pi_T^\nu$ from~Lem.~\ref{lm:change-of-Pmu-Pnu}.  Denote the Radon-Nikodyn (R-N) derivative as
\[
\gamma_T(x):=\frac{\ud \pi_T^\mu}{\ud
	\pi_T^\nu}(x),\quad x\in\bS
\]
It is noted that while $\gamma_0= \frac{\ud \mu}{\ud \nu}$ is a deterministic function on $\bS$, $\gamma_T$ is a $\clZ_T$-measurable function on $\bS$.
A filter is said to be stable if the random function $\gamma_T\to 1$, in a suitable sense, as $T\to\infty$. In this paper, the following notion of filter stability is adopted based on
$\chisq$-divergence\footnote{For any two probability measures $\mu,\nu\in\clP(\bS)$ such that
  $\mu\ll\nu$, the $\chisq$-divergence $\chisq(\mu\mid
                                                    \nu) := \int_\bS
                                                    (\frac{\ud
                                                      \mu}{\ud \nu}(x) -
                                                    1)^2\ud \nu(x)$.}:

\begin{definition}\label{def:filter-stability}
	The nonlinear filter is \emph{stable} in the sense of 
	\begin{align*}
		\text{($\chi^2$ divergence)}\qquad&
                                                    \E^\mu\big(\chisq(\pi_T^\mu\mid \pi_T^\nu)\big) \; \longrightarrow\; 0
	\end{align*}
	as $T\to \infty$ for every $\mu, \nu\in\clP(\bS)$ such that
	$\mu\ll \nu$.
\end{definition}

\begin{remark}\label{rem:var_gammaT_is_chisq}
  $\gamma_T:\bS\to\Re$ is a non-negative random function on $\bS$ with $\E^\nu (\gamma_T(X_T)|\clZ_T)=\pi_T^\nu(\gamma_T) = \int_\bS \gamma_T(x) \ud \pi_T^\nu(x)=1$.  The square of the function $\gamma_T$ is denoted by $\gamma_T^2$ (That is, $\gamma_T^2(x) := (\gamma_T(x))^2$ for $x\in\bS$). Then
 \[
\E^\nu (|\gamma_T(X_T)-1|^2|\clZ_T) = 
  \pi_T^\nu (\gamma_T^2) -1  = \chisq\big(\pi_T^\mu\mid\pi_T^\nu\big)
\]
Therefore, $\chi^2$-divergence $\chisq\big(\pi_T^\mu\mid\pi_T^\nu\big)$ has the meaning of the conditional variance of the random variable $\gamma_T(X_T)$.  Next, $\E^\nu(\gamma_T(X_T)) = \E^\nu (\E^\nu (\gamma_T(X_T)|\clZ_T)) = 1$ and therefore the variance of $\gamma_T(X_T)$ is given by,
\[
\var^\nu (\gamma_T(X_T)) = \E^\nu (|\gamma_T(X_T)-1|^2)
\]
\end{remark}

\section{Backward map for the nonlinear filter}\label{sec:bmap}


A key original concept introduced in this paper is the \emph{backward map}
$\gamma_T\mapsto y_0$ defined as follows:
\begin{equation}\label{eq:bmap}
	y_{0}(x) :=  \E^\nu (\gamma_T(X_T)|[X_0=x]),\quad x\in\bS
\end{equation}
The function $y_0:\bS\to \Re$ is deterministic, non-negative, and 
$\nu(y_0) =  \E^\nu (\gamma_T(X_T)) = 1$.

Since $\mu\ll\nu$, it follows $\mu(y_0) = \E^\mu\big(\gamma_T(X_T)\big)$. Using the tower property, 
\[
\mu(y_0) = \E^\mu\big(\gamma_T(X_T)\big) = \E^\mu\big(\E^\mu(\gamma_T(X_T)|\clZ_T)\big) = \E^\mu\big(\pi_T^\mu(\gamma_T)\big) 
\]
Now, $\pi_T^\mu(\gamma_T)=\pi_T^\nu(\gamma_T^2)$, and therefore,
\[
\mu(y_0)=\E^\mu\big(\pi_T^\nu(\gamma_T^2)\big)
\]
Noting $\pi_T^\nu(\gamma_T^2)-1 = \chisq\big(\pi_T^\mu\mid\pi_T^\nu\big)$ is the $\chisq$-divergence (see formula in Rem.~\ref{rem:var_gammaT_is_chisq}), 
\[
\E^\mu\big(\chisq\big(\pi_T^\mu\mid\pi_T^\nu\big)\big) = \mu(y_0) -
\nu(y_0)
\]
Therefore, filter stability in the sense of $\chisq$-divergence is
equivalent to showing that $\mu(y_0) \stackrel{(T\to
  \infty)}{\longrightarrow} 1$. 

Because 
$
\mu(y_0)-\nu(y_0) = \nu\big((y_0-1)(\gamma_0-1)\big)
$, 
upon using the Cauchy-Schwarz inequality, 
\begin{equation}\label{eq:chisq_identity_intro}
	|\E^\mu\big(\chi^2(\pi_T^\mu\mid\pi_T^\nu)\big)|^2\leq
\var^\nu (y_0(X_0)) \; \chisq(\mu | \nu)
\end{equation}
where $\var^\nu(y_0(X_0))=\E^\nu\big( |y_0(X_0)-1|^2\big)$. 



From~\eqref{eq:chisq_identity_intro}, for all such choices of priors $\mu\in\clP(\bS)$ such that $\chi^2( \mu|\nu)
<\infty$, 
a sufficient condition for filter stability is the following:
\begin{flalign}\label{eq:VDP}
	&\textbf{(variance decay prop.)}\quad  \var^\nu\big(y_0(X_0)\big)  \stackrel{(T\to\infty)}{\longrightarrow}
	0 &
\end{flalign}

We have thus shown the following:

\medskip

\begin{proposition}[\textbf{Answer to Q1 in \Sec{sec:intro}}]\label{prop:ans_Q1}
Consider the backward map $\gamma_T\mapsto y_0$ defined
by~\eqref{eq:bmap}. Suppose $\chi^2( \mu|\nu)
<\infty$ and the variance decay property~\eqref{eq:VDP} holds.  Then
the filter is stable in the sense of $\chisq$-divergence. 
\end{proposition}

\medskip

Next, from~\eqref{eq:bmap}, $
(y_0(X_0) -1) =  \E^\nu \big((\gamma_T(X_T)-1)|X_0\big)$, 
and using
Jensen's inequality, 
\begin{equation}\label{eq:jensens_ineq}
	\var^\nu\big(y_0(X_0)\big) \le \var^\nu\big(\gamma_T(X_T)\big)
\end{equation}
where $\var^\nu(\gamma_T(X_T)):=\E^\nu\big(|\gamma_T(X_T)-1|^2\big)$. 
Therefore, the backward map $\gamma_T\mapsto y_0$ is non-expansive: The variance of the random variable $y_0(X_0)$ is no larger than the variance of the random variable $\gamma_T(X_T)$.

Because contractive operators are a subset of non-expansive operators, one may ask if filter stability is obtained from showing that the backward map $\gamma_T\mapsto y_0$ is contractive?  The answer is provided in the following proposition:

\medskip

\begin{proposition}\label{prop:stronger_jensen_implies_filter_stability}
Suppose a stronger form of~\eqref{eq:jensens_ineq} holds s.t.
\begin{equation}\label{eq:jensens_ineq_stronger}
\text{(Assumption)} \quad \var^\nu\big(y_0(X_0)\big) \le e^{-cT} \var^\nu\big(\gamma_T(X_T)\big) 
\end{equation}
(Because of~\eqref{eq:jensens_ineq}, this is always true with $c=0$). Then
\[
\E^\mu\big(\chi^2(\pi_T^\mu\mid\pi_T^\nu)\big) \leq \frac{1}{\underline{a}} e^{-cT} \chisq(\mu|\nu)
\]
where $\underline{a} = \text{essinf}_{x\in\bS} \gamma_0(x)$. 
\end{proposition}
\medskip
\begin{proof}
See Appendix~\ref{apdx:pf-jensen-and-stability}. 
\end{proof}

\medskip

Based on the backward map, the analysis of filter stability involves
the following objectives:
\begin{enumerate}
\item To justify the stronger form~\eqref{eq:jensens_ineq_stronger} under a suitable definition of the Poincar\'e inequality (PI). 
\item Relate PI to the model properties, namely,
  (i) ergodicity of the Markov process; and (ii) 
  observability/detectability of the HMM.
\end{enumerate} 
While the general case remains open, these objectives are described for the special class of
HMMs with white noise observations.
 
The following remark is included to help relate the approach of this
paper to the
literature.  The reader may choose to skip ahead to
\Sec{sec:white-noise-model} without any loss of continuity. 


\medskip

\begin{remark}[Comparison with the forward map]\label{rm:forward-map}
The backward map is contrasted with the \emph{forward map}, which is the
starting point of the intrinsic approach to the problem of filter
stability~\cite{chigansky2009intrinsic}.  The \emph{forward map} $\gamma_0\mapsto \gamma_T$ is
defined as follows:
\begin{equation*}\label{eq:fmap}
	\gamma_T(x) = \E^\nu
	\Big( \frac{\gamma_0(X_0)}{\E^\nu(\gamma_0(X_0) \mid \clZ_T)}
        \,\bigg|\,\clZ_T\vee [X_T=x]\Big),\quad x\in \bS
\end{equation*}
Upon using the map to express the total variation, one can show (see~\cite[Sec.~6.5]{JinPhDthesis} for a complete derivation), 
\begin{align*}
\lim_{T\to \infty} & \E^\mu\big(\|\pi_T^\mu-\pi_T^\nu\|_\tv\big) = \\
&\E^\nu\Big(\,\Big|\E^\nu\big(\gamma_0 (X_0)\mid \bigcap_{T\ge 0}\clZ_\infty\vee \clF^X_{[T,\infty)}\big) - \E^\nu\big(\gamma_0 (X_0)\mid \clZ_\infty\big)\Big|\, \Big)
\end{align*}
where $\clZ_\infty=\bigcup_{T>0} \clZ_T$, $\clF^X_{[T,\infty)} =
\sigma\big(\{X_t: t \ge T\}\big)$ is the tail sigma-algebra of the
state process $X$. 
As a function of $T$, $\clZ_\infty\vee \clF^X_{[T,\infty)}$ is a
decreasing filtration and $\clZ_T$ is an increasing filtration. Therefore, by the martingale convergence theorem, both terms on the right-hand side converge as $T\to \infty$.  
The limit is zero if the following tail sigma-field identity holds:
\begin{equation*}\label{eq:kunita-tail-sigmafield}
	\bigcap_{T\ge 0} \clZ_\infty \vee \clF^X_{[T,\infty)} \stackrel{?}{=} \clZ_\infty 
\end{equation*}
This identity is referred to as the central problem in the stability
analysis of the nonlinear filter~\cite{van2010nonlinear}. The problem
generated significant attention
(see~\cite{budhiraja2003asymptotic} and references therein).

\end{remark}

\section{Embedding the backward map in a BSDE}\label{sec:white-noise-model}

\subsection{White noise observation model}

In the remainder of this paper, the observation process $Z$ is
according to 
the stochastic differential equation (SDE): 
	\begin{equation}\label{eq:obs-model}
		Z_t = \int_0^t h(X_s) \ud s + W_t,\quad t \ge 0
	\end{equation}
	where $h:\bS\to \Re^m$ is referred to as the observation function and $W =
	\{W_t:0\le t \le T\}$ is an $m$-dimensional Brownian motion
	(B.M.). We write $W$ is $\sP$-B.M. 
	It is assumed that $W$ is independent of $X$. 

For the ensuing analysis, we also need to specify additional notation
for the Markov process $X$.  Specifically, the infinitesimal generator
of the Markov process $X$ is denoted by $\clA$.  In terms of $\clA$, an important
definition is as follows:

\begin{definition}[Defn.~1.4.1. in~\cite{bakry2013analysis}] 
\label{def:cdc}
The
  bilinear operator
\[
\Gamma(f, g)(x) := (\clA fg)(x) - f(x)(\clA g)(x) - g(x)(\clA f)(x),\;x\in\bS
\]
defined for every $(f,g)\in\clD \times \clD$ is called the \emph{carr\'e du
champ operator} of the Markov generator $\clA$. Here, $\clD$ is a 
vector space of (test) functions that are dense in a suitable $L^2$ space, stable under
products (i.e., $\clD$ is an algebra), and $\Gamma:\clD\times
\clD\to \clD$ such that $\Gamma(f,f)\geq 0$ for every
$f\in\clD$. For the case where an invariant measure $\bmu\in\clP(\bS)$ is available then the natural $L^2$ space is with respect to the invariant measure: $L^2(\bmu) = \{f:\bS\to\Re: \bmu(f^2)<\infty\}$. We use the notation $(\Gamma f):= \Gamma(f,f)$.   
\end{definition} 

\medskip

The above is referred to as the \emph{white noise observation model}
of nonlinear filtering. The model is denoted by $(\clA,h)$.  

Because these were stated piecemeal, the main assumptions 
are stated as follows:

\newP{Assumption 1} Consider HMM $(\clA,h)$.  
\begin{enumerate}
\item $X$ is a Feller-Markov process with generator $\clA$ and carr\'e du
champ $\Gamma$. 
\item $Z$ is according to the SDE~\eqref{eq:obs-model} such that the
  Novikov's condition holds:
$
\E \left(\exp\big(\half \int_0^T |h(X_t)|^2\ud t\big)\right) < \infty
$. 
\item $\mu,\nu\in\clP(\bS)$ are two priors with $\mu\ll\nu$. 
\end{enumerate}

For the HMM $(\clA,h)$, the 
nonlinear filter solves the
celebrated \emph{Kushner-Stratonovich equation}: 
\begin{equation}\label{eq:Kushner}
	\ud \pi_t(f) = \pi_t(\clA f) \ud t +
        \big(\pi_t(hf)-\pi_t(f)\pi_t(h)\big)^\tp \ud I_t
\end{equation}
where the \emph{innovation process} is defined by
\[
I_t := Z_t - \int_0^t \pi_s(h)\ud s,\quad t \ge 0
\]
With $\pi_0=\rho\in\clP(\bS)$, the filter $\{\pi_t^\rho:0\leq t\leq T\}$ is the solution
of~\eqref{eq:Kushner}.  Therefore, for the
HMM $(\clA,h)$, $\gamma_T$ is the R-N ratio of the
solution of~\eqref{eq:Kushner}, $\pi_T^\mu$ and $\pi_T^\nu$, with the two choices of priors,
$\pi_0=\mu$ and $\pi_0=\nu$, respectively.

\subsection{Embedding the backward map in a BSDE}

We continue the analysis of the backward
map $\gamma_T \mapsto y_0$ introduced as~\eqref{eq:bmap} 
in~\Sec{sec:bmap}.  For this purpose, consider the backward
stochastic differential equation (BSDE):
\begin{align}
	-\ud Y_t(x) &= \big((\clA Y_t)(x) + h^\tp(x)V_t(x)\big)\ud t -
                      V_t^\tp(x) \ud Z_t, \nonumber\\
	\quad Y_T(x) &= \gamma_T(x),\;\; x\in\bS, \;\;0\leq t\leq T \label{eq:closed-loop-bsde-2}
\end{align}
Here $(Y,V) = \{(Y_t(x),V_t(x)):\Omega \to \Re\times \Re^m
\;:\; x\in\bS, \;0\leq t\leq T\}$ is a $\clZ$-adapted solution of the
BSDE for a prescribed $\clZ_T$-measurable terminal condition
$Y_T=\gamma_T=\frac{\ud 
		\pi^\mu_T}{\ud \pi^\nu_T}$. 

For the HMM $(\clA,h)$, the relationship between the BSDE~\eqref{eq:closed-loop-bsde-2} and the backward map~\eqref{eq:bmap} is given
of the following proposition:

\medskip

\begin{proposition}[\textbf{Answer to Q2 in \Sec{sec:intro}}]
  \label{prop:backward-map-and-bsde}
Fix $T$.  Suppose $Y_0$ is the solution of~\eqref{eq:closed-loop-bsde-2} at time $t=0$ and $y_0$ is defined according to the backward map~\eqref{eq:bmap}. Then
\begin{equation*}\label{eq:bmap-and-bsde}
Y_0(x)=y_0(x) ,\quad x \in \bS
\end{equation*}
and along the solution $(Y,V)$ of~\eqref{eq:closed-loop-bsde-2},  
\begin{equation}\label{eq:var_contractive_ext}
\frac{\ud }{\ud t} \var^\nu (Y_t(X_t)) = \E^\nu \big(
  \pi_t^\nu(\Gamma Y_t) +
  \pi_t^\nu(|V_t|^2) \big),\quad 0\leq t\leq T
\end{equation}
where $\var^\nu(Y_t(X_t)) :=\E^\nu(|Y_t(X_t)-1|^2)$ and $
\pi_t^\nu(|V_t|^2):=\int_\bS V^\tp_t(x) V_t(x)\ud
\pi_t^\nu(x)$. (Eq.~\eqref{eq:var_contractive_ext} is an example of a
dissipation equation, and therefore an answer to 
question Q2 in \Sec{sec:intro}.) 
\end{proposition} 

\medskip

\begin{proof}
See Appendix~\ref{apdx:pf-prop-bmap-bsde}.
\end{proof}
 
\medskip
 
\begin{remark}[Relationship to~\eqref{eq:jensens_ineq_stronger}]
Based on the dissipation equation~\eqref{eq:var_contractive_ext} in \Prop{prop:backward-map-and-bsde}, in order to obtain variance decay, a natural assumption is as follows:
\begin{align*}
  \text{(Assmp.)} \;\; \E^\nu \big( \pi_t^\nu(\Gamma Y_t) + \pi_t^\nu(|V_t|^2) \big) \geq c \; \var^\nu(Y_t(X_t)),\;
  0\leq t\leq T
\end{align*}
From \eqref{eq:var_contractive_ext} then $
\frac{\ud }{\ud t} \var^\nu (Y_t(X_t)) \geq c\,  \var^\nu(Y_t(X_t))
$. Because $Y_T(X_T)=\gamma_T(X_T)$ at the terminal time, Gronwall implies $\var^\nu (Y_0(X_0)) \le e^{-cT} \var^\nu(\gamma_T(X_T))$ from which filter stability follows (see~\Prop{prop:stronger_jensen_implies_filter_stability}).  In a prior conference paper~\cite{kim2021ergodic}, conditional Poincar\'e inequality (c-PI) is introduced for which the assumption above can be verified for $\nu=\bmu$.  Several examples of Markov processes are described for which the c-PI holds.  An important such example is the case where the Markov process satisfies the Doeblin condition. 
\end{remark}

\medskip

Our aim in the remainder of this paper is to define appropriate
notions of energy, variance, and Poincar\'e Inequality (PI) for HMM
(Defn.~\ref{def:PI_HMM}) and relate the PI to the model properties (\Sec{sec:model_props}). 

\section{Poincar\'e Inequality (PI) for HMM}\label{sec:main-results}

\subsection{Function spaces and notation}

Let $\rho\in\clP(\bS)$ and $\tau>0$.  These are used to denote a
generic prior and a generic time-horizon $[0,\tau]$. (In the analysis
of filter stability, these are fixed to $\rho=\nu$ and $\tau=T$).  The
space of Borel-measurable deterministic functions is denoted
\[
L^2(\rho)=\{f:\bS\to \Re \;:\;\rho(f^2) = \int_\bS |f(x)|^2 \ud\rho (x) <\infty\}
\]

\newP{Background from nonlinear filtering} 
A standard approach 
is based upon the Girsanov change of measure. Because the Novikov's
condition holds, define a new measure $\tsP^\rho$
on $(\Omega,\clF_\tau)$ as follows:
\[
\frac{\ud \tsP^\rho}{\ud \sP^\rho} = \exp\Big(-\int_0^\tau
h^\tp(X_t) \ud W_t - \half \int_0^\tau |h(X_t)|^2\ud t\Big) =: D_\tau^{-1}
\]
Then the probability law for $X$ is unchanged but
$Z$ is a $\tsP^\rho$-B.M.~that is independent of $X$~\cite[Lem.~1.1.5]{van2006filtering}. The
expectation with respect to $\tsP^\rho$ is denoted by
$\tE^\rho(\cdot)$.  The unnormalized filter $\sigma_\tau^\rho(f) :=
\tE^\rho(D_\tau f(X_\tau)|\clZ_\tau)$ for $f\in C_b(\bS)$.  It is
called as such because $\pi_\tau^\rho(f) = \frac{\sigma_\tau^\rho(f)}{\sigma_\tau^\rho(\ones)}$.

In a prior work, we introduced  a dual optimal control formulation of the nonlinear filter~\cite{kim2019duality,duality_jrnl_paper_II}.  This requires consideration of the following Hilbert spaces:

\newP{$\bullet$ Hilbert space for the dual}  Formally, the ``dual'' is a function on the state-space $\bS$.  The space of such functions is denoted as
$\clY$.  For the case when
$\bS=\{1,2,\hdots,d\}$, $\clY=\Re^d$.  Related to the dual, two types
of Hilbert spaces are of interest.  These are defined
as follows:
\begin{itemize}
\item Hilbert space of $\clZ_\tau$-measurable random functions:
\begin{align*}
\mathbb{H}_\tau^\rho & :=\{ F:\Omega\to \clY \; : \; F\in\clZ_\tau \;\;
                       \& \;\; \tE^\rho
  (\sigma_{\tau}^\rho (F^2)) < \infty\} 
\end{align*}
(This function space is important because 
the backward map~\eqref{eq:bmap} 
is a map from
$\gamma_T \in \mathbb{H}_T^\nu$ to $y_0\in L^2(\nu)$). 
\item Hilbert space of $\clY$-valued $\clZ$-adapted stochastic processes:
\begin{align*}
\mathbb{H}^\rho([0,\tau]) & :=\{ Y:\Omega\times [0,\tau] \to \clY  :
                            \;  Y_t\in\clZ_t, \;0\leq t\leq \tau, \\
& \qquad\qquad 
                       \& \;\; \tE^\rho \left( \int_0^\tau 
  \sigma_{t}^\rho (Y_t^2) \ud t \right) < \infty\}
\end{align*}
(This function space is important because the solution $Y$ of the BSDE is an element of $\mathbb{H}^\rho([0,\tau])$). 
\end{itemize}

\newP{Notation} 
Let $\rho\in{\cal P}(\bS)$.  
For real-valued functions $f,g\in \clY $, $
\clV_t^\rho(f,g) := \pi_t^\rho\big((f-\pi_t^\rho(f))(g-\pi_t^\rho(g))\big)
$. 
With $f=g$, $\clV_t^\rho(f) := \clV_t^\rho(f,f)$.


\subsection{Definitions of energy, variance, and PI}
\label{ssec:BSDE-formulae}




\newP{Dual optimal control system}
\begin{align}
	-\!\ud Y_t(x) &= \big((\clA Y_t)(x) -  h^\tp (x) \clV_t^\rho(h,Y_t)
	\nonumber \\ &\;\;\;\; + h^\tp (x)(V_t(x) -
	\pi_t^\rho(V_t))\big)\ud t - V_t^\tp(x)\ud Z_t,\;\;0\leq t\leq \tau \nonumber \\
	Y_\tau (x)  & = F(x), \;\; x \in \bS \label{eq:optimal_control_system}
\end{align}
$(Y,V) \in \mathbb{H}^\rho([0,\tau])\times \mathbb{H}^\rho([0,\tau])^m$ is the solution of~\eqref{eq:optimal_control_system} for a given $F\in \mathbb{H}_\tau^\rho$.
The dual optimal control system is important because of the following
relationship to the nonlinear filter:

\medskip

\begin{proposition}[Prop.~1 in~\cite{duality_jrnl_paper_II}]
\label{prop:bsde-general-F}
Consider~\eqref{eq:optimal_control_system}. Then for a.e.~$t\in[0,\tau]$,
\begin{subequations}
	\begin{align}
		&	\pi_t^\rho (Y_t) = \rho (Y_0) +\int_0^t \big(\clV_s^\rho(h,Y_s)+
		\pi_s^\rho(V_s) \big)^\tp\ud Z_s,\quad\sP^\rho-\text{a.s.}\label{eq:estimator-t} \\
		&	\E^\rho\big(\clV_t^\rho (Y_t)\big) = \label{eq:estimator-t-variance}\\
&\quad	\VV
		 + \E^\rho\Big(\int_0^t \pi_s^\rho(\Gamma Y_s) + |\clV_s^\rho(h,Y_s)|^2 +
		 \clV_s^\rho(V_s)\ud s\Big)  \nonumber
		\end{align}
\end{subequations}
\end{proposition}

\medskip

\begin{remark}\label{rem:BSDE_embedding_OCS}
The BSDE embedding~\eqref{eq:closed-loop-bsde-2} of the backward map~\eqref{eq:bmap}
is a special case of~\eqref{eq:optimal_control_system}. In particular,
with $F=\gamma_T$, using~\eqref{eq:estimator-t} with $\tau=T$ and $\rho=\nu$,
\[
\clV_t^\nu(h,Y_t)+ \pi_t^\nu(V_t) = 0,\quad \sP^\nu\text{-a.s.}, \; 0\le t\le T
\]
(because $\pi_T^\nu(\gamma_T)=1$). 
Therefore,~\eqref{eq:optimal_control_system} reduces
to~\eqref{eq:closed-loop-bsde-2}. 
\end{remark}

\medskip

Let $
\clN:= \{\rho\in\clP(\bS): \VV=0 \;\;\forall F\;\in \mathbb{H}_\tau^\rho\}
$.

\medskip

\begin{definition}\label{def:PI_HMM}
Consider~\eqref{eq:optimal_control_system}.  {\em Energy} is defined as follows:
\begin{align*}
\II^\rho(F)&:= \E^\rho \left( \int_0^\tau \pi_t^\rho(\Gamma Y_t) + |\clV_t^\rho(h,Y_t)|^2 +
  \clV_t^\rho(V_t) 
             \ud t \right)
\end{align*}
For $\rho\in\clP(\bS)\setminus \clN$, consider 
\[
\beta^\rho:= \inf\big\{\II^\rho(F) \; : \;F\in \mathbb{H}_\tau^\rho \;\; \& \;\;\VV=1\big\}
\]
and the {\em Poincar\'e constant} is defined as follows:
\[
c^\rho:=\begin{cases}
	\dfrac{1}{\tau}\log\big(1+\beta^\rho\big),\quad &\rho\in\clP(\bS)\setminus \clN\\
	0,\quad & \rho\in\clN
\end{cases}
\]
\end{definition}

\medskip

\begin{remark}\label{rm:meaning-of-crho}
The reason for defining the Poincar\'e constant in this manner is that $c^\rho$ then represents a rate.  In particular, using~\eqref{eq:estimator-t-variance}, for each $\rho \in \clP(\bS)\setminus \clN$,
\[
\VV \le e^{-\tau \, c^\rho} \E^\rho\big(\clV_\tau^\rho(F)\big),\quad \forall\, F\in \mathbb{H}_\tau^\rho 
\]
\end{remark}

\subsection{Existence of minimizer}

We are interested in existence of the minimizers of
the energy functional $\II^\rho(F)$ for $F\in \mathbb{H}_\tau^\rho$.  If it exists, a minimizer is not unique
because of the following translation symmetry:
\[
\II^\rho(F+ \alpha\ones) = \II^\rho (F)
\]
for any $\clZ_\tau$-measurable random variable $\alpha$ such that
$\tE^\rho (\alpha^2)<\infty$.
For this reason, consider the subspace
\[
{\cal S}^\rho := \{F\in \mathbb{H}_\tau^\rho \;:\;
\pi_\tau^\rho(F)=0, \;\; \sP^\rho-a.s.\}
\]
Then $\cal S^\rho$ is closed subspace. (Suppose $F^{(n)} \to F$ in
$\mathbb{H}_\tau^\rho$ with $\pi^\rho_\tau\big(F^{(n)}\big) = 0$. Then   
$\E^\rho\big(|\pi^\rho_\tau(F)|\big) = \E^\rho\big(|\pi^\rho_\tau(F-F^{(n)})|\big) \leq
\E^\rho\big(\pi^\rho_\tau(|F-F^{(n)}|^2)\big) = \tE^\rho\big(\sigma_\tau^\rho
(|F-F^{(n)}|^2)\big) = \| F-F^{(n)} \|_{\mathbb{H}_\tau^\rho}\to 0$.). 

\medskip

\begin{proposition}\label{prop:BSDE_on_S}
Consider the optimal control system~\eqref{eq:optimal_control_system} with $Y_T=F\in{\cal S}^\rho$. Then $\rho(Y_0) = 0$ and
\[
\clV_t^\rho(h,Y_t)+ \pi_t^\rho(V_t) = 0,\quad \sP^\rho\text{-a.s.}, \; 0\le t\le \tau
\]
\end{proposition}

\medskip

\begin{proof}
The result follows from using~\eqref{eq:estimator-t} in
Prop.~\ref{prop:bsde-general-F} (similar to Rem.~\ref{rem:BSDE_embedding_OCS}). 
\end{proof}

\medskip

Because of Prop.~\ref{prop:BSDE_on_S}, for $Y_T=F\in{\cal
  S}^\rho$,~\eqref{eq:optimal_control_system} simplifies to 
\begin{align}
	-\ud Y_t(x) &= \big((\clA Y_t)(x) + h^\tp(x)V_t(x)\big)\ud t -
                      V_t^\tp(x) \ud Z_t, \;\;0\leq t\leq \tau  \nonumber\\
	\quad Y_\tau &= F \in {\cal S}^\rho,\;\; x\in\bS \label{eq:optimal_control_system_on_S}
\end{align}
Note that this is identical to the BSDE embedding~\eqref{eq:closed-loop-bsde-2} of the backward map.
Its solution is used to define a linear operator as follows:
\begin{align*}
\LL_0&: {\cal S}^\rho \subset 
\mathbb{H}_\tau^\rho \to  L^2(\rho) \quad \text{by}\quad \LL_0(F) := Y_0
\end{align*}
(It is noted that~\eqref{eq:optimal_control_system_on_S} and therefore $\LL_0$ do not depend upon $\rho$ even though the optimal control system~\eqref{eq:optimal_control_system} does). 
Additional details concerning this operator appear in
Appendix~\ref{apdx:pf-lemma-minimizer} where it is shown that $\LL_0$ is bounded
with $\|\LL_0\|\leq 1$. 

\medskip

The following Lemma provides sufficient condition for a minimizer to exist:
\begin{lemma}\label{lm:minimizer}
Let $\rho \in{\cal P}(\bS) \setminus \clN$. Suppose that $\LL_0$ is compact. Then
there exists an ${F} \in {\cal S}^\rho $ such that 
\[
\beta^\rho = \II^\rho({F}) \quad\text{and}\quad \VV=1
\]
\end{lemma}

\medskip

\begin{proof}
See Appendix~\ref{apdx:pf-lemma-minimizer}.
\end{proof}

\section{PI and HMM model properties}\label{sec:model_props}

In this section, we make the following assumption:

\newP{Assumption 2}  The state-space is finite:
\begin{align*}
	\textbf{(A2)}\qquad	\bS = \{1,2,\hdots,d\}
\end{align*}

\newP{Notation}
The space of functions and measures are both identified with $\Re^d$: a real-valued function $f$ (resp., a measure $\mu$) is identified with a column vector in $\Re^d$ where the $x^{\text{th}}$ element of the vector equals $f(x)$ (resp., $\mu(x)$) for $x\in\bS$, and $\mu(f) = \mu^\tp f$.  In this manner, the
observation function $h:\bS\to\Re^m$ is also identified with a matrix
$H\in\Re^{d\times m}$.  Its $j$-th column is denoted $H^j$ for $j=1,2,\hdots,m$.  The constant function $\ones = [1,1,\hdots,1]$ is a $d$-dimensional vector with all entries equal to one.  $\clP(\bS)$ is the probability simplex in
$\Re^d$.   The generator $\clA$ is identified with a transition rate matrix, denoted as
$A$, whose $(i,j)$ entry (for $i\neq j$) gives the non-negative rate of transition from state
$i\mapsto j$.  The diagonal
entry $(i,i)$ is chosen such that the sum of the elements in the
$i$-th row is zero.   The finite state-space HMM is denoted as $(A,H)$.  For any function $g:\bS\to\Re$, the notation
\[
hg :=\{\text{diag}(H^j)g : j=1,2,\hdots,m\}
\]  
For $m=1$, this is simply the element-wise multiplication of the function $h$ and $g$ ($hg(x)=h(x)g(x)$ for $x\in\bS$).  

\medskip

\begin{definition} \label{def:obsvbl} Consider an HMM $(A,H)$ on a finite state-space $\bS=\{1,2,\hdots,d\}$. The space of
  {\em observable functions} is the smallest subspace $\clO\subset
  \Re^d$ that satisfies the following two properties:  
	\begin{enumerate}
		\item The constant function $\ones\in \clO$; and
		\item If $g\in\clO$ then $A g \in \clO$ and $hg 
		\in\clO$.
\end{enumerate}
The space of {\em null eigenfunctions} is defined as
\[
S_0 := \{f \in \Re^d \mid \; A f = 0 \} 
\]
\end{definition}

\medskip

\begin{definition}\label{def:ergodic}
Consider an HMM $(A,H)$ on a finite state-space $\bS=\{1,2,\hdots,d\}$.
\begin{enumerate}
\item The HMM is {\em observable} if $\clO = \Re^d$. 
\item The Markov process is {\em ergodic} if
\begin{equation*}\label{eq:PI_MP_0}
A f =0\;\; \implies f= c\ones
\end{equation*}
\item The HMM is {\em detectable} if
$
S_0 \subset \clO
$. 
\end{enumerate}
\end{definition}

\begin{remark}
For additional motivation and background for these definitions, see~\cite[Sec.~IV]{duality_jrnl_paper_I} and~\cite[Ch.~8]{JinPhDthesis}. It is shown (see~\cite[Rem.~13]{duality_jrnl_paper_I}) that the definition is equivalent to the (standard) definition of observability and detectability of HMM introduced in~\cite{van2009observability}. 
\end{remark}

\begin{example}
Consider an HMM on $\bS=\{1,2\}$ with 
\begin{equation*}
A = \begin{bmatrix} -\lambda_{12} & \lambda_{12} \\ \lambda_{21} &    -\lambda_{21} \end{bmatrix},\quad
H =  \begin{bmatrix} h(1) \\
  h(2) \end{bmatrix}
\end{equation*}
For this model, the carr\'e du champ operator and the observable space are as follows:
\[
\Gamma f = \begin{bmatrix} \lambda_{12} \\
    \lambda_{21}\end{bmatrix} (f(1)-f(2))^2,\quad \clO = \text{span} \left\{ \begin{bmatrix} 1 \\ 1 \end{bmatrix}, \begin{bmatrix} h(1)\\h(2) \end{bmatrix}\right\}
\]
Consequently,
\begin{enumerate}
\item $A$ is 
ergodic iff $(\lambda_{12}+\lambda_{21})>0$.  In this case,
the invariant measure $\bmu = \begin{bmatrix}
  \frac{\lambda_{21}}{(\lambda_{12}+\lambda_{21})} & 
  \frac{\lambda_{12}}{(\lambda_{12}+\lambda_{21})} \end{bmatrix}^\transpose$. 
\item $(A,H)$ is  is observable iff $h(1)\neq h(2)$. 
\end{enumerate}
\end{example}

\subsection{Main result}

\begin{proposition}[\textbf{Answer to Q3 in \Sec{sec:intro}}]\label{prop:sufficiency}
Consider the HMM $(A,H)$ on finite state-space and $\rho\in
\clP(\bS) \setminus \clN$.  Suppose any one of
the following conditions holds:
\begin{enumerate}
\item The Markov process is ergodic; or
\item The HMM is observable; or
\item The HMM is detectable.
\end{enumerate}
Then $c^\rho>0$.
\end{proposition}

\begin{proof}
See Appendix~\ref{apdx:pf-main-results-1}.
\end{proof}

\medskip

The converse of this result -- which  gives the tightest condition
for $c^\rho$ to be positive -- is as follows:

\medskip

\begin{proposition}[\textbf{Answer to Q3 in \Sec{sec:intro}}]\label{prop:necessity}
Consider the HMM $(A,H)$ on finite state-space. If $c^\rho >0$ for all $\rho\in
\clP(\bS) \setminus \clN$, then the HMM is detectable.
\end{proposition}

\begin{proof}
See Appendix~\ref{apdx:pf-main-results-2}. 
\end{proof}

\section{Numerical example}\label{sec:numerics}

Consider an HMM on $\bS = \{1,2,3,4\}$ with the transition rate matrix given by
\[
A(\epsilon) = \begin{pmatrix}
	-1 & 1 & 0 & 0 \\ 2 & -2 & 0 & 0 \\ 0 & 0 & -1 & 1 \\ 0 & 0 & 2 & -2
\end{pmatrix}+ \epsilon \begin{pmatrix}
	0 & 0 & 0 & 0 \\ 0 & -1 & 1 & 0 \\ 0 & 1 & -1 & 0 \\ 0 & 0 & 0 & 0
\end{pmatrix}
\]
There are two cases:
\begin{enumerate}
\item \textbf{Case 1:} $\epsilon = 0$. The Markov process is not ergodic.  The space of null eigenfunctions is given by,
\[
S_0 = N(A)= \text{span}\left\{ \begin{pmatrix} 1\\ 1 \\ 0 \\ 0
	\end{pmatrix}, \begin{pmatrix} 0\\ 0 \\ 1 \\ 1
	\end{pmatrix} \right\}
\]
This shows that the subsets $\{1,2\}$ and $\{3,4\}$ are the two communicating classes. 
\item \textbf{Case 2:} $\epsilon > 0$. The Markov process is ergodic with only a single communicating class given by $\bS$. 
\end{enumerate}
Consider three choices for the observation function:
\[
h^1 = \begin{pmatrix} 2\\ 0 \\ 2 \\ 0
	\end{pmatrix},\quad h^2 = \begin{pmatrix} 2\\ 0 \\ 0 \\ 0
	\end{pmatrix},\quad h^3 = \begin{pmatrix} 2\\ 0 \\ -2 \\ 0
	\end{pmatrix}
\]
With $\epsilon>0$, the HMM is detectable for any of these choices.  Therefore, the interesting case arises when $\epsilon=0$.  In this case, the following sub-cases arise:
\begin{enumerate}
\item \textbf{Case 1.1:} $\epsilon=0$ and $H=h^1$. The system $(A,H)$ is not detectable. 
\item \textbf{Case 1.2:} $\epsilon=0$ and $H=h^2$. The system $(A,H)$ is not observable, but it is detectable. 
\item \textbf{Case 1.3:} $\epsilon=0$ and $H=h^3$. The system $(A,H)$ is observable.
\end{enumerate}


\begin{table}\label{tb:parameters}
	\caption{Combinations of simulation parameters}\label{tab:counterexample}
	\centering
	\begin{tabular}{c|c||c|c}
		\hline
		$\epsilon$ & $h$ & Model property & Rate of conv.\\
		\hline
		$0$ & $h^1$ & Not detectable & $0$ \\
		$0$ & $h^2$ & Non-ergodic but detectable & $0.075$ \\
		$0$ & $h^3$ & Observable & $0.155$\\
		$0.1$ & $h^1$ & Ergodic with $h(1)=h(3)$ & $0.196$ \\
		$0.1$ & $h^3$ & Ergodic with $h(1)\neq h(3)$ & $0.412$ \\
		\hline
	\end{tabular}
\end{table}

\begin{figure}
	\includegraphics[width=0.99\linewidth]{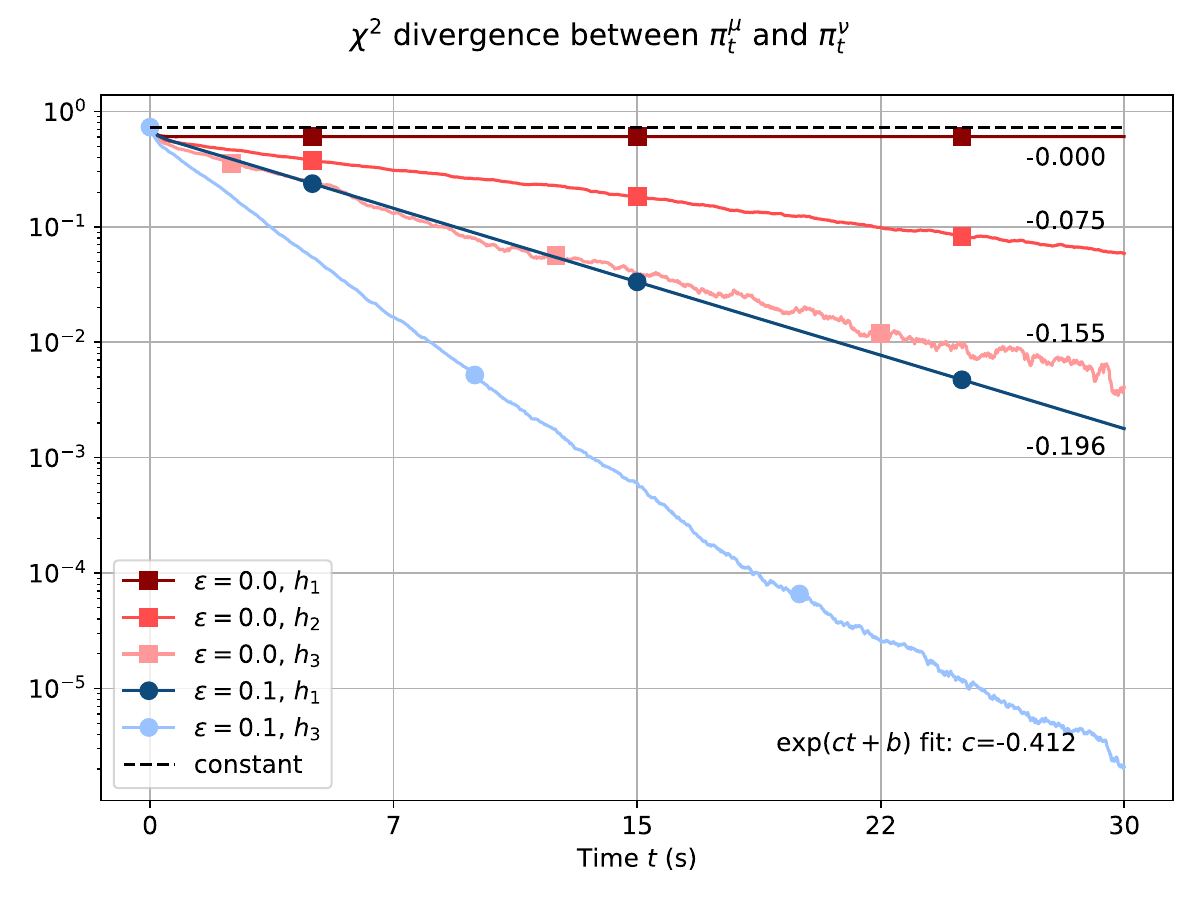}
	\caption{$\E\big(\chisq(\pi_t^\mu\mid\pi_t^\nu)\big)$ for the model with different values of $\epsilon$ and observation function. The number on each line shows the exponential rate obtained from linear fitting.} \label{fig:numerics}
\end{figure}

Fig.~\ref{fig:numerics} depicts the expected value of the $\chisq$-divergence between $\pi_t^\mu$ and $\pi_t^\nu$ from $\mu = [0.25, 0.40, 0.30, 0.05]$ and $\nu = [0.1, 0.2, 0.3, 0.4]$.
The parameters together with the estimated rate of convergence are summarized in Table~\ref{tb:parameters}. An Euler discretization with step-size $0.005$ is used to simulate the HMM and the nonlinear filter. The expectation is approximated by averaging over 500 Monte-Carlo simulations.

\bibliographystyle{IEEEtran}
\bibliography{_master_bib_jin,jin_papers}

\appendix

\section{Appendix}

\subsection{Dissipation equation for a Markov process}\label{apdx:diss_MP}

The semigroup $\{P_t: t\ge 0\}$ is a solution of the 
Kolmogorov equation, $
\frac{\partial}{\partial t} (P_tf) = \clA(P_t f)
$ for $t\geq 0$. Therefore,
\begin{align*}
  \frac{\ud}{\ud t}\clV^\bmu(P_tf)
	&=2\bmu\big((P_tf)(\clA (P_tf))\big) - 2\bmu(P_tf)\bmu\big(\clA(P_tf)\big)\\
	&=-\bmu(\Gamma(P_tf)) 
\end{align*}
where $\bmu\big(\clA(P_tf)\big) = 0$ because $\bmu$ is an invariant measure.   

\subsection{Proof of Proposition~\ref{prop:stronger_jensen_implies_filter_stability}}\label{apdx:pf-jensen-and-stability}

Using inequalities~\eqref{eq:chisq_identity_intro} and~\eqref{eq:jensens_ineq_stronger}, 
\[
|\E^\mu\big(\chi^2(\pi_T^\mu\mid\pi_T^\nu)\big)|^2\le
 e^{-cT} \var^\nu\big(\gamma_T(X_T)\big) \chisq(\mu | \nu) 
\]
Since $\var^\nu\big(\gamma_T(X_T)\big) = \E^\nu\big(|\gamma_T(X_T)-1|^2\big) = \E^\nu\big(\chi^2(\pi_T^\mu\mid\pi_T^\nu)\big)$, 
\[
\E^\mu\big(\chi^2(\pi_T^\mu\mid\pi_T^\nu)\big)\le \frac{1}{R_T}e^{-cT} \chisq(\mu | \nu) 
\]
where 
$
R_T := \frac{\E^\mu\big(\chi^2(\pi_T^\mu\mid\pi_T^\nu)\big)}{\E^\nu\big(\chi^2(\pi_T^\mu\mid\pi_T^\nu)\big)} \geq \text{essinf}_{x\in\bS} \frac{\ud \mu}{\ud \nu}(x) = \underline{a}
$.


\subsection{Proof of Proposition~\ref{prop:backward-map-and-bsde}}\label{apdx:pf-prop-bmap-bsde}

Apply It\^o formula on $Y_t(X_t)$ to obtain
\begin{align*}
	\ud Y_t(X_t) &= V_t^\tp(X_t)\ud W_t + \ud N_t
\end{align*}
where $\{N_t:t\ge 0\}$ is a martingale~\cite[Remark~1]{duality_jrnl_paper_II}. Integrating,
\begin{equation}\label{eq:ytxt}
	\gamma_T(X_T) = Y_t(X_t) + \int_t^T V_s^\tp(X_s)\ud W_s +\ud N_t
\end{equation}
and therefore
\[
Y_t(x) = \E^\nu\big(\gamma_T(X_T)\mid \clZ_t \vee [X_t = x]\big),\quad x\in\bS
\]
In particular at time $t=0$, we have $Y_0(x) = y_0(x)$. The variance of $Y_t(X_t)$ is also obtained from~\eqref{eq:ytxt}:
\begin{align*}
	\E^\nu\big(|&\gamma_T(X_T)-1|^2\big)\\
	&= \E^\nu\Big(|Y_t(X_t)-1|^2 + \int_t^T |V_s(X_s)|^2 + (\Gamma Y_s)(X_s)\ud s\Big)\\
	&=\var^\nu(Y_t(X_t)) + \E^\nu\Big(\int_t^T \pi_s^\nu(\Gamma Y_s) + \pi_s^\nu(|V_s|^2)\ud s\Big)
\end{align*}
Upon differentiating both sides with $t$ gives~\eqref{eq:var_contractive_ext}.

\subsection{Proof of Lemma~\ref{lm:minimizer}}\label{apdx:pf-lemma-minimizer}

\newP{Discussions on the map $\LL_0$ and $\II^\rho$} For $F\in \mathbb{H}_\tau^\rho$, note that
\[
\|F\|_{\mathbb{H}_\tau^\rho}^2 = \tE^\rho\big(\sigma_\tau^\rho(F^2)\big)= \E^\rho\big(\pi_\tau^\rho(F^2)\big) = \E^\rho\big(F(X_\tau)^2\big)
\]
For $F\in \clS^\rho$, $\pi_\tau(F) = 0$ and 
$\therefore, \;
\E^\rho\big((F(X_\tau))^2\big) = \E^\rho\big(\clV_\tau^\rho(F)\big)
$. 
Eq.~\eqref{eq:estimator-t-variance} in Prop.~\ref{prop:bsde-general-F} is thus expressed as
\begin{equation}\label{eq:norm-equation}
\|F\|_{\mathbb{H}_\tau^\rho}^2 = \rho(Y_0^2) + \II^\rho(F),\quad F\in\clS^\rho
\end{equation}
This shows that $\LL_0:\mathbb{H}_\tau^\rho\to L^2(\rho)$ is bounded with $\|\LL_0\|\le 1$.

To obtain the minimizer, setting $(\tilde Y,\tilde V)$ to be the solution to~\eqref{eq:optimal_control_system_on_S} with $Y_\tau = \tilde F\in {\cal
	S}^\rho $, the functional derivative is evaluated as
follows:
\[
\langle \nabla \II^\rho(F) , \tilde F \rangle := 2\;
\E^\rho \left( \int_0^\tau \pi_t^\rho(\Gamma(Y_t,\tilde Y_t) +
\pi_t^\rho( V_t^\tp \tilde V_t) 
\ud t \right)
\]
where note $\pi_t^\rho(V_t^\tp \tilde{V}_t) := \int_\bS V_t^\tp(x) \tilde{V}_t(x) \ud \pi_t^\rho(x) $. Using Cauchy-Schwarz  and~\eqref{eq:norm-equation},
\begin{equation}\label{eq:bdd_lin_fn}
	|\langle \nabla \II^\rho(F) , \tilde F \rangle|^2 \leq 4
	\|F\|_{\mathbb{H}_\tau^\rho}^2
	\; \|\tilde F\|_{\mathbb{H}_\tau^\rho}^2
\end{equation}
This shows that $\tilde F \mapsto \langle \nabla
\II^\rho(F) , \tilde F\rangle$ is a bounded linear functional as a map
from 
${\cal S}^\rho \subset \mathbb{H}_\tau^\rho$ into $\Re$. 
With these formalities completed, we show a minimizer exists.


\newP{Proof of Lemma~\ref{lm:minimizer}} Let $\beta^\rho$ be the infimum.  Consider a sequence $\{F^{(n)}\in \clS^\rho:n=1,2,\ldots\}$ such that $\II^\rho(F^{(n)})\downarrow \beta^\rho$ and $\rho\big((Y_0^{(n)})^2\big)=1$ for each $n$, with $Y_0^{(n)}:=\LL_0(F^{(n)})$. Using~\eqref{eq:norm-equation},
\[
\|F^{(n)}\|_{\mathbb{H}_\tau^\rho}^2 = 1 + \II^\rho(F^{(n)}) < C,\quad n=1,2,\ldots
\]
Therefore, $F^{(n)}$ is a bounded sequence in the Hilbert space $\mathbb{H}_\tau^\rho$, and there exists a weak limit $F\in\mathbb{H}_\tau^\rho$ such that $F^{(n)}\rightharpoonup F$. Moreover, $F\in\clS^\rho$ because $\clS^\rho$ is closed. Therefore $Y_0:=\LL_0(F)$ satisfies $\rho(Y_0)=0$.
Since $F^{(n)} \rightharpoonup F$ in $\mathbb{H}_\tau^\rho$ and $\LL_0$ is compact, we have $Y_0^{(n)}\to Y_0$ in $L^2(\rho)$. Therefore, $\rho\big((Y_0)^2\big) = \lim_{n\to \infty}\rho\big((Y_0^{(n)})^2\big) = 1$.

It remains to show that $\II^\rho(F) = \beta^\rho$.
The map $F\mapsto \II^\rho(F)$ is convex. Therefore,
\[
\II^\rho(F^{(n)})\ge \II^\rho(F) + \langle\nabla \II^\rho(F),F^{(n)}-F\rangle
\]
We have already shown that $\tilde F \mapsto \langle \nabla
\II^\rho(F) , \tilde F\rangle$ is a bounded linear functional.
Therefore, letting $n\to \infty$, the second term on the right-hand
side converges to zero and 
\[
\II^\rho(F) \le \lim_{n\to \infty} \II^\rho(F^{(n)}) = \beta^\rho
\]
Because $\beta^\rho$ is the infimum, this show that $\II^\rho(F) = \beta^\rho$.


\subsection{Proof of Proposition~\ref{prop:sufficiency}}\label{apdx:pf-main-results-1}

\begin{proof}[of Prop.~\ref{prop:sufficiency}]
Suppose any one of the three conditions hold. We claim then
\[
\text{(claim)} \quad \II^\rho(F) = 0 \implies \VV = 0
\]
Supposing the claim is true, the proof is by contradiction.  To see this suppose  $\beta^\rho=0$. By
Lemma~\ref{lm:minimizer}, there exists $\II^\rho(F) = 0$ such that $\VV
=1$ which contradicts the claim.  It remains to prove the claim.  For
the three cases, its proof is described in the remainder of this appendix. 
\end{proof}

\newP{1. Ergodic case}  At time $t>0$, let $\rho_t$ denote the probability law of $X_t$ (without conditioning). Then $\rho\ll\rho_t$ and $\text{supp}(\rho_t) =:\bS'$ is identical for any $t>0$. W.l.o.g., take $\bS'$ as the new state-space and consider the Markov process on $\bS'$.\\
Suppose $\II^\rho(F) = 0$. Then
\[
\E^\rho \Big( \int_0^\tau \pi_t^\rho\big(\Gamma Y_t\big) \ud t \Big) =0 
	\implies
	\pi_t^\rho\big(\Gamma Y_t\big) =0
\]
almost every $t\in[0,\tau]$, $\sP^\rho|_{\clZ_\tau}$-almost
surely. For white noise observation model, $\text{supp}(\pi_t^\rho)
=:\bS'$ and therefore $\Gamma Y_t(x) = 0$ for all $x\in\bS'$, and
therefore $\Gamma Y_t = 0$ with probability 1. If the model is
ergodic, this implies $Y_t$ is a constant function, and therefore
$\E^\rho\big(\clV_t^\rho(Y_t)\big) = 0$. The proof of the claim is completed by noting $\var^\rho(Y_0(X_0))\le \E^\rho(\clV_t^\rho(Y_t))$ from~\eqref{eq:estimator-t-variance}.

\newP{2. Observable case} The proof is based on using the equation for conditional covariance $\clV_t^\rho(f,Y_t)$ (see~\cite[Appdx.~D]{kim2024variance} for its proof):
\begin{align}
	&	\ud \clV_t^\rho (f,Y_t) = \Big(\pi_t^\rho\big(\Gamma(f,Y_t)\big)+\clV_t^\rho(\clA f,Y_t)\Big)\ud t  \label{eq:momentum}\\
	&\quad +\Big(\clV_t^\rho\big((f-\pi_t^\rho(f))(h-\pi_t^\rho(f)),Y_t\big)+\clV_t^\rho(f,V_t)\Big)^\tp \ud I_t^\rho \nonumber
\end{align}
The equation is used to prove the
following Lemma which is the key to prove the claim. 

\medskip

\begin{lemma}\label{Lem:obsvbl}
	Suppose $\II^\rho(F) = 0$.  Then for each $ f\in\clO$,
	\[
	\clV_t^\rho(f,Y_t) = 0,  \quad \sP^\rho\text{-a.s.},\;\; \text{a.e.} \; 0\leq t\leq \tau
	\]
\end{lemma}

\medskip

\begin{proof}
	From the defining relation for $\II^\rho(F)$,
	\[
	\pi_t^\rho (\Gamma Y_t) = 0,\; \clV_t^\rho (h,Y_t) = 0,\;
	\clV_t^\rho (V_t) = 0,\;\;\sP^\rho\text{-a.s.}
	\]
	for $\text{a.e.}\;0\le t\le \tau$.  Using the Cauchy-Schwarz formula
	then for each $ f\in C_b(\bS)$,
	\[
	|\clV_t^\rho (f,V_t)|^2 \leq \clV_t^\rho (f) \clV_t^\rho (V_t) =0 \quad \sP^\rho\text{-a.s.}
	\]
	Similarly, upon using the Cauchy-Schwarz
	formula~\cite[Eq.1.4.3]{bakry2013analysis} for carr\'e du
	champ, 
	$
	\pi_t^\rho (\Gamma(f,Y_t))= 0, \; \sP^\rho\text{-a.s.}
	$. 
	Based on these~\eqref{eq:momentum} simplifies to
	\begin{align*}
		&	\ud \clV_t^\rho (f,Y_t) 
	  \\&\qquad
          =\clV_t^\rho (\clA f,Y_t)\ud t + \big(\clV_t^\rho (hf,Y_t)  -
		\pi_t^\rho (h)\clV_t^\rho (f,Y_t)\big)^\tp \ud I_t^\rho 
	\end{align*}
	Therefore, $\clV_t^\rho (f,Y_t) = 0, \;\;\;0\leq t\leq \tau$
	\begin{align*}
		&\Longrightarrow \quad \clV_t^\rho (\clA f,Y_t) = 0,\;\clV_t^\rho (hf,Y_t) = 0, \;\;\;0\leq t\leq \tau
	\end{align*}
	Since $\clV_t^\rho (\ones,Y_t) = 0$ for all $t\in[0,\tau]$, the result follows
	from Defn.~\ref{def:obsvbl} of the observable space $\clO$.
\end{proof}

Based on the result in Lemma~\ref{Lem:obsvbl}, if $\clO = \Re^d$, we have $\E^\rho\big(\clV_t^\rho(Y_t)\big) = 0$, and then the claim follows because $\var^\rho(Y_0(X_0))\le \E^\rho(\clV_t^\rho(Y_t))$ from~\eqref{eq:estimator-t-variance}.

\newP{3. Detectable case} As in the ergodic case, if $\II^\rho(F) = 0$ then $\Gamma Y_t(x) = 0$ for all $x\in\bS'$, and therefore $Y_t \in S_0$. If the system $(\clA,h)$ is detectable, then this implies $Y_t \in \clO$ with probability 1. By Lemma~\ref{Lem:obsvbl}, $\E^\rho\big(\clV_t^\rho(Y_t)\big) = 0$ and the claim follows.

\subsection{Proof of Proposition~\ref{prop:necessity}}\label{apdx:pf-main-results-2}

Suppose HMM is not detectable. Our goal is to find $\rho\in\clP(\bS) \setminus \clN$ and $F\in\mathbb{H}_\tau^\rho$ such that $\II^\rho(F) = 0$ and $\VV = 1$.
We begin with two claims:
\begin{enumerate}
\item \textbf{Claim 1.} There exists a $\rho\in\clP(\bS)$ and $f\in
  S_0$ such that (a) $\rho(f^2) > 0$, and (b) $\rho(fg)=0$ for all $g\in\clO$.
\item \textbf{Claim 2.} For any such $f$ and $\rho$ (that satisfy the two conditions in claim 1),
$
\pi_t^\rho(fg) \equiv 0,\;0\le t\le \tau
$, $\forall\;g\in\clO$.
\end{enumerate}

Assuming these claims are true $\pi_\tau^\rho(f) = 0$ and therefore $f\in
\clS_0$.  Because $f\in S_0$, $\clA f = 0$ and
therefore, $Y_t \equiv f$ and $V_t \equiv 0$ solves the
BSDE~\eqref{eq:optimal_control_system_on_S}, whose energy
$\II^\rho(f) = 0$ and $\VV = \rho(f^2) > 0$.  It remains to prove the two claims.

\newP{Proof of claim 2} Since $f\in S_0$, $\clA f = 0$, and consequently, $\clA(fg) = f\clA g$. From~\eqref{eq:Kushner}:
	\[
	\ud \pi_t^\rho(fg) = \pi_t^\rho\big(f\clA g\big)\ud t + \pi_t^\rho(ghf)^\tp \ud I_t ,\;0\leq t\leq \tau
	\]
	Because $\clA g	\in \clO$ and $gh \in \clO$ (Defn.~\ref{def:obsvbl}) this shows claim 2.  At time $t=0$, $\left. \pi_t^\rho\big(f\clA g\big)\right|_{t=0} = \rho( f \clA g) =0$ and $\left. \pi_t^\rho(ghf)\right|_{t=0} = \rho(fgh) = 0$.  Therefore, $\pi_t^\rho(fg) \equiv 0$ is the equilibrium soln. 

        \newP{Proof of claim 1} Consider an ergodic partition $S=\cup_k S^{(k)}$ and the invariant measures $\bmu^{(k)}$ with support on $S^{(k)}$ for each $k$.  W.l.o.g, upon re-ordering of indices if necessary, there exists an $f = \ones_{S^{(1)}} - \ones_{S^{(2)}}$ such that $f\notin\clO$.  Therefore, w.l.o.g, we may restrict ourselves to ergodic partition with exactly two components, $S=  S^{(1)} \cup  S^{(2)}$, with invariant measures $\bmu^{(1)}$ and $\bmu^{(2)}$ with support on $S^{(1)}$ and $S^{(2)}$, respectively.  For this case, $\dim(N(\clA))=2$ and since $\ones$ is contained in both $N(\clA)$ and in $\clO$, and the fact that $N(\clA)\not\subset \clO$, $\dim(N(\clA) \cap \clO)=1$. Consider the restriction (noting $\clO$ is $\clA$-invariant),
        \[
\left. \clA\right|_{\clO}:\clO \to \clA(\clO)\subset \clO
\]
Then $\dim(N(\left. \clA\right|_{\clO}))=1$. By rank-nullity, $\dim(\clA(\clO))=\dim(\clO)-1$.  Consider a decomposition $\clO = \clA(\clO) \oplus \text{span}\{\bar{o}\}$ for some $\bar{o}\in\clO$. Use the decomposition to express any $g\in \clO$ as
$
g = g_A + a\,\bar{o}
$, 
where $a\in\Re$.  Now, $\bmu^{(1)}(\ones)=\bmu^{(2)}(\ones)=1 \implies \ones\notin\clA(\clO)$ (since, for all $g\in\clA(\clO)$, $\bmu^{(1)}(g)=\bmu^{(2)}(g)=0$). Because $\ones \in \clO$, a possible choice is to pick $\bar{o}=\ones$. 

We now pick $\rho \in \clP(\bS)$ and $f \in S_0$ to show that $\rho(fg)=0$ for all $g\in\clO$. Set
\[
\rho = \half \bmu^{(1)} + \half \bmu^{(2)},\quad f = c_1 \ones_{S^{(1)}} + c_2 \ones_{S^{(2)}}
\]
where the constants $c_1$ and $c_2$ need to be picked. Because $g_A\in R(\clA)$,
$
\rho(fg_A) = \frac{1}{2} (c_1 \bmu^{(1)}(g_A) + c_2 \bmu^{(2)}(g_A)) = 0
$. 
Therefore,
\[
\rho(fg) = \frac{a}{2} (c_1 \bmu^{(1)}(\bar{o}) + c_2 \bmu^{(2)}(\bar{o}))
\]
Pick constants $c_1$ and $c_2$ to make the right-hand side zero. With $\bar{o}=\ones$, $c_1=1$ and $c_2=-1$ works for which $\rho(f)=0$ and $\rho(f^2)=1$.

\end{document}